\documentclass{article}
\usepackage[utf8]{inputenc}
\usepackage[margin=1in]{geometry}
\usepackage{amsmath, amssymb, bbm, amsthm, graphicx, etoolbox, subcaption}
\usepackage[hidelinks]{hyperref}
\usepackage[capitalize]{cleveref}

\newcommand{\mc}{\mathcal}

\newtheorem{thm}{Theorem}[section] 

\newtheorem{prop}[thm]{Proposition}
\newtheorem{fact}[thm]{Fact}
\newtheorem{conj}[thm]{Conjecture}
\usepackage{tikz}

\usepackage{biblatex}
\title{Unimodality of a refinement of Lassalle's sequence}
\author{Mihir Singhal\thanks{Massachusetts Institute of Technology, Cambridge, MA 02139. Email: \href{mailto:mihirs@mit.edu}{\nolinkurl{mihirs@mit.edu}}.}}
\date{} 

\begin{document}

\maketitle

\begin{abstract}
Defant, Engen, and Miller defined a refinement of Lassalle's sequence $A_{k+1}$ by considering uniquely sorted permutations of length $2k+1$ whose first element is $\ell$. They showed that each such sequence is symmetric in $\ell$ and conjectured that these sequences are unimodal. We prove that the sequences are unimodal.
\end{abstract}

\section{Introduction}
We study a refinement of Lassalle's sequence introduced by Defant, Engen, and Miller \cite{dem}. Lassalle's sequence was originally defined by Lassalle in \cite{lassalle} by the recurrence
\[A_m = (-1)^{m-1}C_m + \sum_{j-1}^{m-1} (-1)^{j-1} \binom{2m-1}{2m-2j-1}A_{m-j}C_j,\]
with the initial condition $A_1 = 1$, and where the $C_k = \binom{2k}{k}/(k+1)$ are the Catalan numbers. In \cite{lassalle} 
Lassalle proved the sequence had positive terms, and the sequence has been further explored in \cite{cgw, josuat-verges, tevlin, wz}. This sequence also has relations with noncommutative probability: $(-1)^{m-1}A_m$ is the $(2m)$-th classical cumulant of the standard semicircular probability distribution. More details about the connection between noncommutative probability and stack sorting can be found in \cite{defant}.

We are interested in combinatorial interpretations of Lassalle's sequence, particularly those found by Josuat-Verg\`es in \cite{josuat-verges} and by Defant, Engen, and Miller in \cite{dem}.

Defant, Engen, and Miller's interpretation of Lassalle's sequence came chronologically after that of Josuat-Verg\`es, but we will discuss it first. The interpretation involves the \textit{stack-sorting map}, which was originally defined by West \cite{west} as a slight modification of an algorithm originally defined by Knuth \cite{knuth}.
Since we will not end up working directly with this map, we will not fully define the map, instead referring readers to \cite{dem} for the definition. Essentially, the stack-sorting map ``partially sorts" a permutation, in such a way that any permutation can be sorted via enough applications of the stack-sorting map. If $\pi$ is a permutation, let $s(\pi)$ denote its image under stack sorting. We say that a permutation is \textit{uniquely sorted} if there is a unique permutation which stack-sorts to it. That is to say, $\pi$ is uniquely sorted if $s^{-1}(\pi)$ has size 1. (Uniquely sorted permutations have also been studied further in \cite{defant2, mularczyk}.) Defant, Engen, and Miller proved in \cite{dem}
that $A_{k+1}$ counts the total number of uniquely sorted permutations of length $2k+1$. Furthermore, they defined the sequences $(A_{k+1}(\ell))_{1 \le \ell \le 2k+1}$, letting $A_{k+1}(\ell)$ equal the number of uniquely sorted permutations of length $2k+1$ whose first element is $\ell$. Note that the sum of each such sequence is $A_{k+1}$, so these may be regarded as refinements of Lassalle's sequence. They proved that each such sequence is symmetric (in $\ell$), and conjectured that these sequences are unimodal, and furthermore, log-concave. In this paper we will prove the former.
\begin{thm} \label{thm:main}
For every $k$, the sequence $(A_{k+1}(\ell))_{1 \le \ell \le 2k+1}$ is unimodal.
\end{thm}
Using \cref{prop:recursion}, which is a recursion-like identity for a generalization of these sequences, we also (with computer assistance) verify the following.
\begin{prop}
For all $k \le 53$, the sequence $(A_{k+1}(\ell))_{1 \le \ell \le 2k+1}$ is log-concave.
\end{prop}

\section{Orientations of partition crossing graphs}
In this section we will describe Josuat-Verg\`es's interpretation of the Lassalle sequence, which will be useful to us in order to prove \cref{thm:main}. First we will need some definitions.

Let $\mc P(n)$ denote the set of partitions of $\{0, \dots, n-1\}$, and if $n$ is even, also let $\mc M(n) \subset \mc P(n)$ denote the set of matchings on $\{0, \dots, n-1\}$, where a \textit{matching} is just a partition containing only blocks of size 2.

If $\rho \in \mc P(n)$ and $B, B'$ are blocks in $\rho$, then we say that $B$ and $B'$ form a \textit{crossing} if there exist $i, k \in B$ and $j, \ell \in B'$ such that either $i < j < k < \ell$ or $i > j > k > \ell$. If we put the elements of $\{0, \dots, n-1\}$ in order on a circle, and represent each block by the polygon whose vertices are its elements, then two blocks form a crossing exactly when their corresponding polygons intersect. Note that in the special case where $\rho$ is a matching, all its blocks are represented by line segments.

Now define the \textit{crossing graph} $G(\rho)$ of a partition $\rho$ to be the graph whose vertices are the blocks of $\rho$ and with an edge between two blocks if and only if they form a crossing. Define an orientation $r$ of the edges of $G(\rho)$ to be \textit{root-connected} with \textit{root} $B$ if it is acyclic and the block $B \in \rho$ is the only source in the orientation. Equivalently, $r$ is root-connected with root $B$ if it is acyclic and there exists a path from $B$ to every other vertex in $G(\rho)$. (Note that $G(\rho)$ must be connected in order for it to have a root-connected orientation.) Greene and Zaslavsky in \cite{gz} proved that the number of root-connected orientations of $G(\rho)$ with any fixed root is $T_{G(\rho)}(1, 0)$, where $T_{G(\rho)}$ is the Tutte polynomial of $G(\rho)$, defined in \cite{aigner}.

Let $\widetilde{\mc P}(n)$ denote the set of pairs $(\rho, r)$, where $\rho \in \mc P(n)$ and $r$ is a root-connected orientation of $G(\rho)$, where the root is the block containing 0. Similarly let $\widetilde{\mc M}(n)$ denote the subset of elements $(\rho, r)$ of $\widetilde{\mc P}(n)$ such that $\rho\in \mc M(n)$. Josuat-Verg\`es proved that Lassalle's sequence $A_{k+1}$ counts the number of elements of $\widetilde{\mc M}(2k+2)$. Furthermore, Defant, Engen, and Miller proved by bijection that the refinement $A_{k+1}(\ell)$ can also be counted by root-connected orientations of matchings:
\begin{prop}[\cite{dem}]
The number of pairs $(\rho, r)$ of matchings $\rho \in \mc M(2k+2)$ and root-connected orientations $r$ of $G(\rho)$ with root $\{0, \ell\}$ is $A_{k+1}(\ell)$.
\end{prop}
This is the interpretation of the sequence $A_{k+1}(\ell)$ which we will use to prove unimodality.

\section{Proof of unimodality}

We will now prove \cref{thm:main}. First we will need to define a slight generalization of $\widetilde{\mc M}(n)$, where we allow the root to be a set of any size.

For a set $S \subset \{0, \dots, n-1\}$, let $\mc M_S(n)$ denote the set of partitions of $\{0, \dots, n-1\}$ in which one of the blocks is $S$ and all other blocks have 2 elements. Let $\widetilde{\mc M}_S(n)$ denote the set of $(\rho, r)$ such that $\rho \in \mc M_S(n)$ and $r$ is a root-connected orientation of $G(\rho)$ with root $S$. Then, let $A_{k+1}(S)$ be the size of the set $\widetilde{\mc M}_S(|S| + 2k)$. We then have $A_{k+1}(\ell) = A_{k+1}(\{0, \ell\})$.

Let $n = |S| + 2k$. We first note some basic properties of the function $A_{k+1}(S)$. Each such property will be accompanied by a visual depiction of it, where the elements of $\{0, \dots, n-1\}$ are placed (equally spaced and clockwise) on a circle and $S$ is represented by a polygon whose vertices are its elements. First, the property of rotation states that we can rotate the elements of $S$ on the circle without changing $A_{k+1}(S)$.

\begin{fact}[Rotation]
We have, for any integer $r$, $A_{k+1}(S) = A_{k+1}(S + r)$, where $S+r$ denotes elementwise addition of $r$ to $S$ and elements are taken mod $n$.
\end{fact}

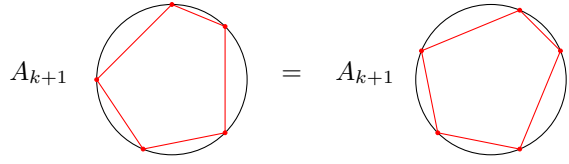
\[
A_{k+1}
\begin{tikzpicture}[baseline=0]
\def \r {1};
\def \s {1.25};
\draw (0, 0) circle (\r);
\draw[red] (90-0*22.5:\r) -- (90-2*22.5:\r) -- (90-6*22.5:\r) -- (90-9*22.5:\r) -- (90-12*22.5:\r) -- (90-0*22.5:\r);
\foreach \d/\l in {{90-0*22.5}/{}, {90-2*22.5}/{}, {90-6*22.5}/{}, {90-9*22.5}/{}, {90-12*22.5}/{}} {
    \draw[fill, red] (\d:\r) circle [radius=0.025];
    \node at (\d:\s) {\l};
}
\end{tikzpicture}
\quad = \quad
A_{k+1} 
\begin{tikzpicture}[baseline=0]
\def \r {1};
\def \s {1.25};
\draw (0, 0) circle (\r);
\draw[red] (90-1*22.5:\r) -- (90-3*22.5:\r) -- (90-7*22.5:\r) -- (90-10*22.5:\r) -- (90-13*22.5:\r) -- (90-1*22.5:\r);
\foreach \d/\l in {{90-1*22.5}/{}, {90-3*22.5}/{}, {90-7*22.5}/{}, {90-10*22.5}/{}, {90-13*22.5}/{}} {
    \draw[fill, red] (\d:\r) circle [radius=0.025];
    \node at (\d:\s) {\l};
}
\end{tikzpicture}
\]

We will generally use the property of rotation implicitly throughout this proof.

The property of merging states that if $S$ contains two consecutive elements then we can merge these two points on the circle into one point, reducing the size of $S$ by 1 and also reducing $n$ by 1. Note that this doesn't affect any crossings. To state this, we will let $S = \{a_1, \dots, a_m\}$, where $m \ge 1$ and $a_1 < \dots < a_m$. We allow indices to ``wrap around" the circle, so that $a_{m+1} = a_1 + n$, and so on. 

\begin{fact}[Merging]
If $a_{j+1} = a_j + 1$, then \[A_{k+1}(\{a_1, \dots, a_m\}) = A_{k+1}(\{a_1, \dots, a_j, a_{j+2}-1, \dots, a_m-1\}).\]
\end{fact}

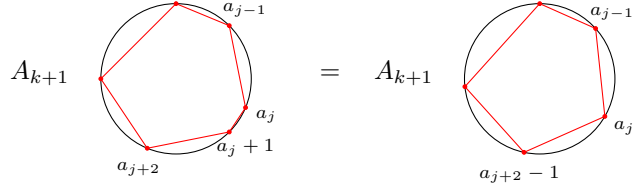
\[
A_{k+1}
\begin{tikzpicture}[baseline=0]
\def \r {1};
\def \s {1.3};
\draw (0, 0) circle (\r);
\draw[red] (90-0*22.5:\r) -- (90-2*22.5:\r) -- (90-5*22.5:\r) -- (90-6*22.5:\r) -- (90-9*22.5:\r) -- (90-12*22.5:\r) -- (90-0*22.5:\r);
\foreach \d/\l in {{90-0*22.5}/{}, {90-2*22.5}/{$a_{j-1}$}, {90-5*22.5}/{$a_j$}, {90-6*22.5}/{$a_j+1$}, {90-9*22.5}/{$a_{j+2}$}, {90-12*22.5}/{}} {
    \draw[fill, red] (\d:\r) circle [radius=0.025];
    \node at (\d:\s) {\scriptsize \l};
}
\end{tikzpicture}
\quad = \quad
A_{k+1}
\begin{tikzpicture}[baseline=0]
\def \r {1};
\def \s {1.3};
\draw (0, 0) circle (\r);
\draw[red] (90-0*24:\r) -- (90-2*24:\r) -- (90-5*24:\r) -- (90-8*24:\r) -- (90-11*24:\r) -- (90-0*24:\r);
\foreach \d/\l in {{90-0*24}/{}, {90-2*24}/{$a_{j-1}$}, {90-5*24}/{$a_j$}, {90-8*24}/{$a_{j+2}-1$}, {90-11*24}/{}} {
    \draw[fill, red] (\d:\r) circle [radius=0.025];
    \node at (\d:\s) {\scriptsize \l};
}
\end{tikzpicture}
\]

Now we show a sort of recursion for $A_{k+1}(S)$, which will be the main tool that we will use.

\begin{prop} \label{prop:recursion}
Suppose $m \ge 2$. Suppose that for some $j$, $a_j \le a_{j+1} - 2$. Then, we have the following identity.
\begin{multline}
A_{k+1}(\{a_1, \dots, a_{j-1}, a_j + 1, a_{j+1}, \dots, a_m\}) - A_{k+1}(\{a_1, \dots, a_{j-1}, a_j, a_{j+1}, \dots, a_m\}) \\
= \sum_{a_j < b < a_{j+1}-1} A_k(\{a_1, \dots, a_{j-1}, a_j, b, a_{j+1}-1, \dots, a_m-1\})  \\
- \sum_{a_{j-1} < b < a_j} A_k(\{a_1, \dots, a_{j-1}, b, a_j, a_{j+1}-1, \dots, a_m-1\}) \label{eq:recursion}
\end{multline}
\begin{multline*}
A_{k+1}
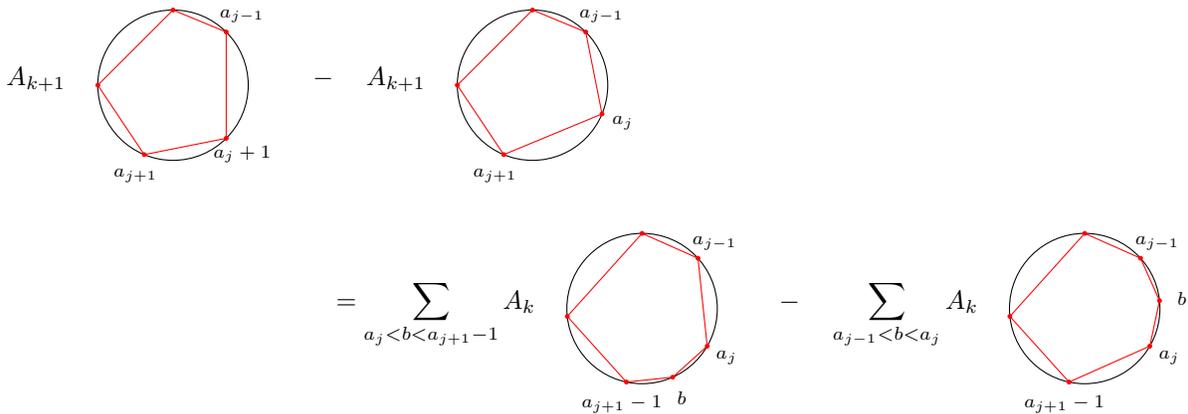
\begin{tikzpicture}[baseline=0]
\def \r {1};
\def \s {1.3};
\draw (0, 0) circle (\r);
\draw[red] (90-0*22.5:\r) -- (90-2*22.5:\r) -- (90-6*22.5:\r) -- (90-9*22.5:\r) -- (90-12*22.5:\r) -- (90-0*22.5:\r);
\foreach \d/\l in {{90-0*22.5}/{}, {90-2*22.5}/{$a_{j-1}$}, {90-6*22.5}/{$a_j+1$}, {90-9*22.5}/{$a_{j+1}$}, {90-12*22.5}/{}} {
    \draw[fill, red] (\d:\r) circle [radius=0.025];
    \node at (\d:\s) {\scriptsize \l};
}
\end{tikzpicture}
\quad - \quad A_{k+1}
\begin{tikzpicture}[baseline=0]
\def \r {1};
\def \s {1.3};
\draw (0, 0) circle (\r);
\draw[red] (90-0*22.5:\r) -- (90-2*22.5:\r) -- (90-5*22.5:\r) --  (90-9*22.5:\r) -- (90-12*22.5:\r) -- (90-0*22.5:\r);
\foreach \d/\l in {{90-0*22.5}/{}, {90-2*22.5}/{$a_{j-1}$}, {90-5*22.5}/{$a_j$}, {90-9*22.5}/{$a_{j+1}$}, {90-12*22.5}/{}} {
    \draw[fill, red] (\d:\r) circle [radius=0.025];
    \node at (\d:\s) {\scriptsize \l};
}
\end{tikzpicture} \\
= \sum_{a_j < b < a_{j+1}-1} A_k
\begin{tikzpicture}[baseline=0]
\def \r {1};
\def \s {1.3};
\draw (0, 0) circle (\r);
\draw[red] (90-0*24:\r) -- (90-2*24:\r) -- (90-5*24:\r) -- (90-6.5*24:\r) -- (90-8*24:\r) -- (90-11*24:\r) -- (90-0*24:\r);
\foreach \d/\l in {{90-0*24}/{}, {90-2*24}/{$a_{j-1}$}, {90-5*24}/{$a_j$},  {90-6.5*24}/{$b$}, {90-8*24}/{$a_{j+1}-1$}, {90-11*24}/{}} {
    \draw[fill, red] (\d:\r) circle [radius=0.025];
    \node at (\d:\s) {\scriptsize \l};
}
\end{tikzpicture}
\quad - \quad
\sum_{a_{j-1} < b < a_j} A_k
\begin{tikzpicture}[baseline=0]
\def \r {1};
\def \s {1.3};
\draw (0, 0) circle (\r);
\draw[red] (90-0*24:\r) -- (90-2*24:\r) -- (90-3.5*24:\r) -- (90-5*24:\r) -- (90-8*24:\r) -- (90-11*24:\r) -- (90-0*24:\r);
\foreach \d/\l in {{90-0*24}/{}, {90-2*24}/{$a_{j-1}$}, {90-5*24}/{$a_j$},  {90-3.5*24}/{$b$}, {90-8*24}/{$a_{j+1}-1$}, {90-11*24}/{}} {
    \draw[fill, red] (\d:\r) circle [radius=0.025];
    \node at (\d:\s) {\scriptsize \l};
}
\end{tikzpicture}
\end{multline*}
\end{prop}
\begin{proof}
Let $S = \{a_1, \dots, a_{j-1}, a_j, a_{j+1}, \dots, a_m\}$ and $S' = \{a_1, \dots, a_{j-1}, a_j+1, a_{j+1}, \dots, a_m\}$. We define a correspondence between some elements of $\widetilde{\mc M}_S(n)$ and some elements of $\widetilde{\mc M}_{S'}(n)$ (where $n = m + 2k$ as usual). To find the corresponding element to any $(\rho, r) \in \widetilde{\mc M}_S(n)$, just swap $a_j$ with $a_{j}+1$ in $\rho$. Keep all orientations the same in $r$. Note that this may create or remove an edge with $S$ (which becomes $S'$); if a new edge is created, direct it away from $S'$ so that $S'$ is still a source. This correspondence is defined for all $(\rho, r)$ such that $S'$ is still the only source in the resulting partition and orientation. Note that this correspondence is injective since $S$ must be a source in $r$, so edges leaving $S$ can be deleted without loss of information. The inverse of this correspondence can be obtained by the exact same process. The correspondence is illustrated in \cref{fig:corresp}.

\begin{figure} 
\centering
\[
\begin{tikzpicture}[baseline=0]
\def \r {1};
\def \s {1.3};
\draw (0, 0) circle (\r);
\draw[red] (90-0*22.5:\r) -- (90-2*22.5:\r)  -- (90-5.5*22.5:\r) -- (90-9*22.5:\r) -- (90-12*22.5:\r) -- (90-0*22.5:\r);
\draw[blue] (90-3.5*22.5:\r) -- (90-6*22.5:\r);
\foreach \d/\l in {{90-0*22.5}/{}, {90-2*22.5}/{$a_{j-1}$}, {90-9*22.5}/{$a_{j+1}$}, {90-12*22.5}/{}} {
    \draw[fill, red] (\d:\r) circle [radius=0.025];
    \node at (\d:\s) {\scriptsize \l};
}
\draw[fill, red] (90-5.5*22.5:\r) circle [radius=0.025];
\node at (90-5.2*22.5:\s) {\scriptsize $a_j$};
\foreach \d/\l in {{90-6*22.5}/{$a_j+1$}, {90-3.5*22.5}/{$b$}} {
    \draw[fill, blue] (\d:\r) circle [radius=0.025];
    \node at (\d:\s) {\scriptsize \l};
}
\end{tikzpicture}
\longleftrightarrow
\begin{tikzpicture}[baseline=0]
\def \r {1};
\def \s {1.3};
\draw (0, 0) circle (\r);
\draw[red] (90-0*22.5:\r) -- (90-2*22.5:\r)  -- (90-6*22.5:\r) -- (90-9*22.5:\r) -- (90-12*22.5:\r) -- (90-0*22.5:\r);
\draw[blue] (90-3.5*22.5:\r) -- (90-5.5*22.5:\r);
\foreach \d/\l in {{90-0*22.5}/{}, {90-2*22.5}/{$a_{j-1}$}, {90-6*22.5}/{$a_j+1$}, {90-9*22.5}/{$a_{j+1}$}, {90-12*22.5}/{}} {
    \draw[fill, red] (\d:\r) circle [radius=0.025];
    \node at (\d:\s) {\scriptsize \l};
}
\foreach \d/\l in {{90-3.5*22.5}/{$b$}} {
    \draw[fill, blue] (\d:\r) circle [radius=0.025];
    \node at (\d:\s) {\scriptsize \l};
}
\draw[fill, blue] (90-5.5*22.5:\r) circle [radius=0.025];
\node at (90-5.2*22.5:\s) {\scriptsize $a_j$};
\end{tikzpicture}
\] 
\caption{The (partial) correspondence between $\widetilde{\mc M}_{S}(n)$ and $\widetilde{\mc M}_{S'}(n)$. This removes a crossing when $a_{j-1} < b < a_j$, and creates a crossing when $a_j+1 < b < a_{j+1}$.}
\label{fig:corresp}
\end{figure}
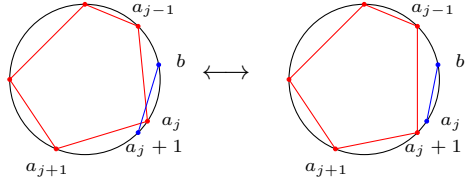

Let $M$ be the set of $(\rho, r) \in \widetilde{\mc M}_{S}(n)$ with no corresponding element in $\widetilde{\mc M}_{S'}(n)$, and let $M'$ similarly be the set of $(\rho, r) \in \widetilde{\mc M}_{S'}(n)$ with no corresponding element in $\widetilde{\mc M}_{S}(n)$. Note that the left-hand side of \eqref{eq:recursion} is just equal to $|M'| - |M|$.

We first count the size of $M$. The only way for $(\rho, r)$ to be an element of $M$ is if switching $a_j$ and $a_j+1$ in $\rho$ causes an edge to be deleted, and as a result there is a new source other than $S'$. The deleted edge must have been from $S$ to a vertex of the form $\{b, a_j+1\}$, for $a_{j-1} < b < a_j$ (and conversely, any edge of such a form will have been deleted when switching $a_j, a_j+1$). Thus, the elements of $M$ are exactly those $(\rho, r)$ where $\{b, a_j+1\}$ is a block in $\rho$ for some $a_{j-1} < b < a_j$ and where all edges are directed away from $\{b, a_j+1\}$ except the one from $S$. 

Now, note that a block forms a crossing with $S$ or $\{b, a_j+1\}$ exactly when it forms a crossing with $S \cup \{b, a_j+1\}$. (This is evident visually from \cref{fig:corresp}.)
Thus, contracting the edge between $S$ and $\{b, a_j+1\}$ in $G(\rho)$ gives the crossing graph of the partition $\rho'$ obtained by combining $S$ and $\{b, a_j+1\}$ in $\rho$. Since $S$ and $\{b, a_j+1\}$ are the only two sources in the orientation $r$ (disregarding the edge between the two), the corresponding orientation $r'$ of $\rho'$ is also acyclic and has $S \cup \{b, a_j+1\}$ as the only source. Thus, $(\rho', r') \in \widetilde{\mc M}_{S \cup \{b, a_j+1\}}(n)$. We can recover $(\rho, r)$ from $(\rho', r')$ by replacing $S \cup \{b, a_j+1\}$ with $S$ and $\{b, a_j+1\}$, directing all edges from each of them outward, and directing the edge between them away from $S$. Thus, we have a bijection between $M$ and the union of $\widetilde{\mc M}_{S \cup \{b, a_j+1\}}(n)$ over all $b$, so
\begin{align*}
|M| &= \sum_{a_{j-1} < b < a_j} |\widetilde{\mc M}_{S \cup \{b, a_j+1\}}(n)| \\
&= \sum_{a_{j-1} < b < a_j} A_k(S \cup \{b, a_j+1\}) \\
&= \sum_{a_{j-1} < b < a_j} A_k(\{a_1, \dots, a_{j-1}, b, a_j, a_j+1, a_{j+1}, \dots, a_m\}) \\
&= \sum_{a_{j-1} < b < a_j} A_k(\{a_1, \dots, a_{j-1}, b, a_j, a_{j+1}-1, \dots, a_m-1\}) & \text{(by the merging property).}
\end{align*}
By an identical argument, we have
\begin{align*}
|M'| &= \sum_{a_j+1 < b < a_{j+1}} |\widetilde{\mc M}_{S \cup \{a_j, b\}}(n)| \\
&= \sum_{a_j+1 < b < a_{j+1}} A_k(S \cup \{a_j, b\}) \\
&= \sum_{a_j+1 < b < a_{j+1}} A_k(\{a_1, \dots, a_{j-1}, a_j, a_j+1, b, a_{j+1}, \dots, a_m\}) \\
&= \sum_{a_j < b < a_{j+1}-1} A_k(\{a_1, \dots, a_{j-1}, a_j, b, a_{j+1}-1, \dots, a_m-1\}).
\end{align*}
Thus, $|M'| - |M|$ equals the right hand side of \eqref{eq:recursion}, as desired.
\end{proof}

We now use \cref{prop:recursion} to prove one final property, which is that we can reflect $a_j$ over the midpoint of $a_{j-1}, a_{j+1}$. As a consequence, $A_{k+1}(S)$ depends only on the (unordered) multiset of differences $\{a_2-a_1, a_3-a_2, \dots, a_{m+1}-a_m\}$ (recall that the last difference here equals $(a_1+n) - a_m$).

\begin{fact}[Reflection]
Suppose $m \ge 2$. Then, for all $j$, \[A_{k+1}(\{a_1, \dots, a_{j-1}, a_j, a_{j+1}, \dots, a_m\}) = A_{k+1}(\{a_1, \dots, a_{j-1}, a_{j+1}+a_{j-1}-a_j, a_{j+1}, \dots, a_m\}).\]
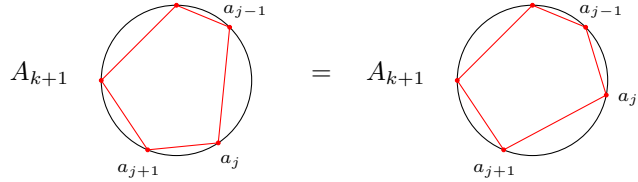
\[
A_{k+1}
\begin{tikzpicture}[baseline=0]
\def \r {1};
\def \s {1.3};
\draw (0, 0) circle (\r);
\draw[red] (90-0*22.5:\r) -- (90-2*22.5:\r) -- (90-6.5*22.5:\r) -- (90-9*22.5:\r) -- (90-12*22.5:\r) -- (90-0*22.5:\r);
\foreach \d/\l in {{90-0*22.5}/{}, {90-2*22.5}/{$a_{j-1}$}, {90-6.5*22.5}/{$a_j$}, {90-9*22.5}/{$a_{j+1}$}, {90-12*22.5}/{}} {
    \draw[fill, red] (\d:\r) circle [radius=0.025];
    \node at (\d:\s) {\scriptsize \l};
}
\end{tikzpicture}
\quad = \quad A_{k+1}
\begin{tikzpicture}[baseline=0]
\def \r {1};
\def \s {1.3};
\draw (0, 0) circle (\r);
\draw[red] (90-0*22.5:\r) -- (90-2*22.5:\r) -- (90-4.5*22.5:\r) -- (90-9*22.5:\r) -- (90-12*22.5:\r) -- (90-0*22.5:\r);
\foreach \d/\l in {{90-0*22.5}/{}, {90-2*22.5}/{$a_{j-1}$}, {90-4.5*22.5}/{$a_j$}, {90-9*22.5}/{$a_{j+1}$}, {90-12*22.5}/{}} {
    \draw[fill, red] (\d:\r) circle [radius=0.025];
    \node at (\d:\s) {\scriptsize \l};
}
\end{tikzpicture}
\]
\end{fact}

\begin{proof}
We induct on $k$ and $a_j$. The $k=0$ case is obvious, and the $a_j = a_{j-1} + 1$ case follows from the merging property. Thus, suppose that $k > 0, a_j > a_{j-1}+1$ and assume the statement is true for $k-1$ and, fixing $k$, also assume it is true for $a_j-1$. We then apply \cref{prop:recursion} to $\{a_1, \dots, a_{j-1}, a_j-1, a_{j+1}, \dots, a_m\}$ and $\{a_1, \dots, a_{j-1}, a_{j+1}+a_{j-1}-a_j, a_{j+1}, \dots, a_m\}$ to get the following:

\begin{align}
&A_{k+1}(\{a_1, \dots, a_{j-1}, a_j, a_{j+1}, \dots, a_m\}) \nonumber\\
&\qquad \qquad = A_{k+1}(\{a_1, \dots, a_{j-1}, a_j-1, a_{j+1}, \dots, a_m\}) \nonumber\\
&\qquad \qquad \qquad + \sum_{a_j-1 < b < a_{j+1}-1} A_k(\{a_1, \dots, a_{j-1}, a_j-1, b, a_{j+1}-1, \dots, a_m-1\})  \nonumber\\
&\qquad \qquad \qquad - \sum_{a_{j-1} < b < a_j-1} A_k(\{a_1, \dots, a_{j-1}, b, a_j-1, a_{j+1}-1, \dots, a_m-1\}), \label{eq:refl1}
\end{align} \\
\begin{align}
&A_{k+1}(\{a_1, \dots, a_{j-1}, a_{j+1}+a_{j-1}-a_j, a_{j+1}, \dots, a_m\}) \nonumber\\
&\qquad \qquad = A_{k+1}(\{a_1, \dots, a_{j-1}, a_{j+1}+a_{j-1}-a_j+1, a_{j+1}, \dots, a_m\}) \nonumber\\
&\qquad \qquad \qquad - \sum_{a_{j+1}+a_{j-1}-a_j < b < a_{j+1}-1} A_k(\{a_1, \dots, a_{j-1}, a_{j+1}+a_{j-1}-a_j, b, a_{j+1}-1, \dots, a_m-1\})  \nonumber\\
&\qquad \qquad \qquad + \sum_{a_{j-1} < b < a_{j+1}+a_{j-1}-a_j} A_k(\{a_1, \dots, a_{j-1}, b, a_{j+1}+a_{j-1}-a_j, a_{j+1}-1, \dots, a_m-1\}). \label{eq:refl2}
\end{align}

By the inductive hypothesis (for $a_j-1$), we have
\begin{align*}
A_{k+1}(\{a_1, \dots, a_{j-1}, a_j-1, a_{j+1}, \dots, a_m\}) = A_{k+1}(\{a_1, \dots, a_{j-1}, a_{j+1}+a_{j-1}-a_j+1, a_{j+1}, \dots, a_m\}),
\end{align*}
and by two successive applications each of the inductive hypothesis (for $k-1$), we also have
\begin{align*}
&A_k(\{a_1, \dots, a_{j-1}, a_j-1, b, a_{j+1}-1, \dots, a_m-1\}) \\
&\qquad = A_k(\{a_1, \dots, a_{j-1}, b+a_{j-1}-a_j+1, b, a_{j+1}-1, \dots, a_m-1\}) \\
&\qquad = A_k(\{a_1, \dots, a_{j-1}, b+a_{j-1}-a_j+1, a_{j+1}+a_{j-1}-a_j, a_{j+1}-1, \dots, a_m-1\}), \\
&A_k(\{a_1, \dots, a_{j-1}, b, a_j-1, a_{j+1}-1, \dots, a_m-1\}) \\
&\qquad = A_k(\{a_1, \dots, a_{j-1}, b, b+a_{j+1}-a_j, a_{j+1}-1, \dots, a_m-1\}) \\
&\qquad = A_k(\{a_1, \dots, a_{j-1}, a_{j+1}+a_{j-1}-a_j, b+a_{j+1}-a_j, a_{j+1}-1, \dots, a_m-1\}).
\end{align*}
Thus, we can match the terms in the right hand side of \eqref{eq:refl1} with those of \eqref{eq:refl2}, so 
\[A_{k+1}(\{a_1, \dots, a_{j-1}, a_j, a_{j+1}, \dots, a_m\}) = A_{k+1}(\{a_1, \dots, a_{j-1}, a_{j+1}+a_{j-1}-a_j, a_{j+1}, \dots, a_m\}),\]
as desired.
\end{proof}

We now show the main theorem. In fact we will show the following which is a generalization of \cref{thm:main}.

\begin{thm}
Let $m \ge 2$, and fix all $a_i$ except $a_j$. Then, $A_{k+1}(\{a_1, \dots, a_{j-1}, a_j, a_{j+1}, \dots, a_m\})$ forms a symmetric and unimodal sequence in $a_j$, as $a_j$ ranges in the interval $(a_{j-1}, a_{j+1})$.
\end{thm}

\begin{proof}
Symmetry follows from the reflection property. We now induct on $k$; the $k=0$ case is obvious because there is only one possibility for $a_j$. Thus assume the statement is true for $k-1$.

Without loss of generality, by symmetry assume $a_{j+1}-a_j \ge a_j - a_{j-1}$; we will show that if $a_j \ge a_{j-1}+2$, then
\[A_{k+1}(\{a_1, \dots, a_{j-1}, a_j, a_{j+1}, \dots, a_m\}) \ge A_{k+1}(\{a_1, \dots, a_{j-1}, a_j-1, a_{j+1}, \dots, a_m\}).\]
Indeed, by \cref{prop:recursion}, we have
\begin{align*}
&A_{k+1}(\{a_1, \dots, a_{j-1}, a_j, a_{j+1}, \dots, a_m\}) - A_{k+1}(\{a_1, \dots, a_{j-1}, a_j-1, a_{j+1}, \dots, a_m\}) \\
={}& \sum_{a_j-1 < b < a_{j+1}-1} A_k(\{a_1, \dots, a_{j-1}, a_j-1, b, a_{j+1}-1, \dots, a_m-1\})  \\
&\qquad - \sum_{a_{j-1} < b < a_j-1} A_k(\{a_1, \dots, a_{j-1}, b, a_j-1, a_{j+1}-1, \dots, a_m-1\}) \\
\ge{}& \sum_{0 < c < a_j - a_{j-1} - 1} \Big(A_k(\{a_1, \dots, a_{j-1}, a_j-1, a_{j+1}-1-c, a_{j+1}-1, \dots, a_m-1\}) \\
&\qquad \qquad \qquad \qquad - A_k(\{a_1, \dots, a_{j-1}, a_{j-1}+c, a_j-1, a_{j+1}-1, \dots, a_m-1\})\Big) \\
\ge{}& \sum_{0 < c < a_j - a_{j-1} - 1} \Big(A_k(\{a_1, \dots, a_{j-1}, a_j-1, a_{j+1}-1-c, a_{j+1}-1, \dots, a_m-1\}) \\
&\qquad \qquad \qquad \qquad - A_k(\{a_1, \dots, a_{j-1}, a_{j}-1-c, a_j-1, a_{j+1}-1, \dots, a_m-1\})\Big) \\
\ge{}& \sum_{0 < c < a_j - a_{j-1} - 1} \Big(A_k(\{a_1, \dots, a_{j-1}, a_j-1, a_{j+1}-1-c, a_{j+1}-1, \dots, a_m-1\}) \\
&\qquad \qquad \qquad \qquad - A_k(\{a_1, \dots, a_{j-1}, a_{j}-1-c, a_{j+1}-1-c, a_{j+1}-1, \dots, a_m-1\})\Big) \\
\ge{}& \sum_{0 < c < a_j - a_{j-1} - 1} 0 \\ ={}& 0,
\end{align*}
where the last inequality is by the inductive hypothesis, using the fact that $a_j-1$ is closer to the center of the interval $(a_{j-1}, a_{j+1}-1-c)$ than $a_j-1-c$. (This is because $(a_j-1)-a_{j-1} > (a_j-1-c)-a_{j-1}$ and $(a_{j+1}-1-c) - (a_j-1) > (a_j-1-c)-a_{j-1}$, where the latter follows from the assumption that $a_{j+1}-a_j \ge a_j - a_{j-1}$.)

Thus the induction is complete and we are done.
\end{proof}

\section{Conclusion}
The main remaining open question is Defant, Engen, and Miller's conjecture of log-concavity:
\begin{conj}[\cite{dem}] \label{conj:main}
The sequence $(A_{k+1}(\ell))_{1 \le \ell \le 2k+1}$ is log-concave for every nonnegative integer $k$.
\end{conj}
It turns out that the recursion \cref{prop:recursion} actually allows us to more efficiently compute the sequence elements $A_{k+1}(\ell)$. Using this recursion, we used a computer program to verify that \cref{conj:main} is true for all $k \le 53$. However, the problem remains open for large $k$.

\section{Acknowledgments}
This research was conducted at the University of Minnesota, Duluth REU and was supported by NSF-DMS grant 1949884 and NSA grant H98230-20-1-0009. Thanks to Joe Gallian for running the REU program and providing helpful comments. Thanks also to Colin Defant for suggesting the problem and for useful comments.

\printbibliography[heading=bibintoc]
\end{document}